\documentclass[leqno,10pt]{amsart}
\usepackage{amsfonts}
\usepackage[english]{babel}
\usepackage{graphicx}
\usepackage{amscd,color}
\usepackage{amsmath}
\usepackage{amssymb}
\usepackage[colorlinks=true,citecolor=blue,linkcolor=blue]{hyperref}
\usepackage[T1]{fontenc}
\usepackage{float} %figura
\hypersetup{%hypertex,
colorlinks=true,linkcolor=blue,filecolor=blue,%pagecolor=blue,
urlcolor=blue}

\makeatletter
\@namedef{subjclassname@2020}{%
\textup{2020} Mathematics Subject Classification}
\makeatother

\newcommand\restr[2]{{
  \left.\kern-\nulldelimiterspace
  #1
  \vphantom{\big|}
  \right|_{#2}
  }}

\theoremstyle{plain}

\newtheorem{example}{\bf Example}
\newtheorem{lemma}{\bf Lemma}

\newtheorem{proposition}{\bf Proposition}
\newtheorem{remark}{Remark}

\newtheorem{theorem}{\bf Theorem}

\numberwithin{equation}{section}

\title[Gap results and existence of CMC free boundary hypersurfaces in rotational domains]{Gap results and existence of CMC free boundary hypersurfaces in rotational domains}

\author[A. Freitas, M.S. Santos and J.S. Sindeaux]{Allan Freitas$^{1, 2, \ast}$,  M\'{a}rcio S. Santos$^1$ and Joyce S. Sindeaux$^1$}

\address{$^1$Departamento de Matem\'{a}tica\\
Universidade Federal da Para\'{\i}ba\\
58.051-900 Jo\~{a}o Pessoa, Para\'{\i}ba, Brazil}
\address{$^2$Università degli Studi di Torino, Dipartimento di Matemática ``Giuseppe Peano'', Torino, TO, Italy.}
\email{allan@mat.ufpb.br}
\email{marcio.santos@academico.ufpb.br}
\email{joycessaraiva@gmail.com}



\thanks{$^\ast$Corresponding author.}

\begin{document}

\begin{abstract}
In this paper, we work with the existence and uniqueness of free boundary constant mean curvature hypersurfaces in rotational domains. These are domains whose boundary is generated by a rotation of a graph. We classify the CMC free boundary hypersurfaces as topological disks or annulus under some conditions on the function that generates the graph and a gap condition on the umbilicity tensor. Also, we construct some examples of free boundary minimal surfaces in the rotational ellipsoid that, in particular, satisfy our gap condition.\end{abstract}

\maketitle

\tableofcontents

\section{Introduction}

% Constant mean curvature (CMC) hypersurfaces are one of the most studied objects in differential geometry. (\textcolor{red}{talk more about CMC})

% Let $\Sigma$ be a two-sided $n$-dimensional hypersurface isometrically immersed in a 
% Riemannian manifold $M^{n+1}$,
% and let denote by $N$ a unit normal vector field globally defined along $\Sigma^n$. The extrinsic geometry
% of $\Sigma$ is naturally measured by the first variation of $N$, given by its shape operator. 
% More precisely, given $p\in \Sigma$, the shape operator of $\Sigma$ with respect to $N$ at $p$ is the 
% self-adjoint linear operator given by 
% $A:T_p\Sigma\to T_p\Sigma$, $A(X) = -\nabla_XN$, where $\nabla$ stands for the 
% Levi-Civita connection on $M$. The eigenvalues of $A$ are the principal curvatures of $\Sigma$ in $M$.

In this work, we consider an $n$-dimensional hypersurface $\Sigma$ with a smooth boundary that is compact, oriented, and immersed in a Riemannian manifold $M^{n+1}$ with a smooth boundary $\partial M$, such that $\partial\Sigma \subset \partial M$. In this setting, we denote by $N$ a unit normal vector field globally defined along $\Sigma$ and the shape operator of $\Sigma$ with respect to $N$ at $p$ is the self-adjoint linear operator $A:T_{p}M\longrightarrow T_{p}M$ given by  $A(X) =-\nabla_{X}N$,
where $\nabla$ stands for the Levi-Civita connection on $M$. The eigenvalues of A
are the principal curvatures of $\Sigma$ in $M$ and the mean curvature is $H=\frac{\mbox{trace}A}{n}$. Additionally, we are interested in  {\it free boundary CMC hypersurface in $M$}: when $\Sigma$ has constant mean curvature (CMC in short), and the boundary of $\Sigma$ meets the boundary of $M$ orthogonally. Such hypersurfaces are stationary for the area functional for variations preserving the enclosed volume (see, for example, \cite[Section 1]{Ros:1995}).
In particular, when $H=0$, we say that $\Sigma$ is a free boundary minimal hypersurface. 
% The relevance of this concept is due, for instance, to the fact that, according to
% Ros and Vergasta in free boundary CMC surfaces are critical points of the area functional for volume-preserving variations.
In the particular case where the domain $M$ is the unitary ball $\mathbb{B}^3$ in the Euclidean space, the simplest examples of CMC free boundary surfaces are the equatorial disk, the critical catenoid (minimal surfaces) and the spherical caps. 

We remember that after its initial motivation by Courant \cite{courant1940} and preliminary developments (for example, \cite{nitsche1985},  \cite{struwe1984} and \cite{Ros:1995}), this topic has received plenty of attention, mainly after the 2010 decade and the undeniable contributions of Fraser and Schoen (\cite{fraser2011}, \cite{fraser2016}). These underlying works reveal, in particular, several similarities between free boundary minimal surfaces in an Euclidean unit ball and closed minimal surfaces in the sphere. In this sense, the classical results and strategies to obtain rigidity results in the last one could indicate the direction of interest in getting related developments in the free boundary case. % any stable immersed CMC hypersurface with free boundary in a Euclidean ball is a totally geodesic ball or a spherical cap (see \cite{wang-xia}).

In this direction, the so-called gap results in the second fundamental form give an important characterization of CMC surfaces in the sphere. A series of contributions from Simons \cite{Simons}, Lawson \cite{L}, and Chern, do Carmo and Kobayashi \cite{CdCK} get the following gap result for the second fundamental form $A$ of the immersion:

\begin{theorem}[Chern-do Carmo-Kobayashi \cite{CdCK}, Lawson \cite{L}, Simons \cite{Simons}] \label{thm:Simons}
Let $\Sigma$ be a closed minimal hypersurface in the unit sphere $\mathbb{S}^{n+1}$. Assume that the second fundamental form $A$ on $\Sigma$ satisfies
$$|A|^2 \leq n.$$
Then
\begin{enumerate}
  \item either $|A|^2=0$ and $\Sigma$ is an equator;
  \item or $|A|^2=n$ and $\Sigma$ is a Clifford minimal hypersurface.
\end{enumerate}
\end{theorem}

In the study of CMC hypersurfaces in the sphere, Alencar and do Carmo \cite{hilario} also obtained a gap result, but now considering the umbilicity tensor $\phi=A-Hg$.

\begin{theorem}[Alencar-do Carmo \cite{hilario}]
Let $\Sigma$ be a closed, CMC hypersurface in the unit sphere $\mathbb{S}^{n+1}$. If
$$\|\phi\|^2\leq C_H,$$

\begin{enumerate}
\item either $\|\phi\|^2\equiv 0$ and $\Sigma^n$ is totally umbilical in $\mathbb{S}^{n+1}$,

\item  $\|\phi\|^2\equiv C_H$ and $\Sigma^n$ is an $H(r)$-torus in $\mathbb{S}^{n+1}$.
\end{enumerate}
Here, $C_{H}$ is related to a root of a polynomial whose coefficients depend on the mean curvature $H$ and the dimension $n$\footnote{For details, see the Introduction of \cite{hilario}.}.
\end{theorem}

We could describe some contributions by starting from these two characterizations and studying similar phenomena for free boundary CMC hypersurfaces in the ball. In \cite{Ambrozio:2021}, Ambrozio and Nunes proved that if $\Sigma$ is a  compact free boundary minimal surface in $\mathbb{B}^3$ and for all points $x$ in $\Sigma$,
\begin{equation}\label{gapambnunes}
|A|^2(x)\langle x, N(x)\rangle ^2\leq 2,    
\end{equation}
then $\Sigma$ is a flat equatorial disk or a critical catenoid. In higher dimensions, some similar gap results to \eqref{gapambnunes} can be obtained for $2$-dimensional surfaces in the ball (see \cite{barbosaviana1}) and, with a topological rigidity, in submanifolds of any codimension in higher dimensional balls (see \cite[Theorem 3.7]{barbosaviana2}). Also, some gaps result just in the second fundamental form, as that in Theorem \ref{thm:Simons}, was obtained in \cite{CMV}, \cite{Barbosa:2021} and 
\cite{BarbosaFreitas:2023}.

The question arises: Does an analogous result hold to these in the context of free boundary CMC non-minimal surfaces? Barbosa, Cavalvante, and Pereira answered this question in \cite{Barbosa:2023}. More specifically, in this work they proved that if $\Sigma$ is a compact free boundary CMC surface in $\mathbb{B}^3$ and for all points $x$ in $\Sigma$,
$$|\phi|^2\langle \vec{x}, N\rangle ^2\leq\frac{1}{2}(2+2H\langle \vec{x},N\rangle ^2),$$
then $\Sigma$ is a totally umbilical disc or a part of a Delaunay surface\footnote{In \cite{Barbosa:2023} the gap appears in a different way because the authors consider the non-normalized mean curvature.}.

In \cite{Bettiol}, Bettiol, Piccione and Santoro have studied the existence of CMC disks and Delaunay annuli that are free boundary in a ball. Furthermore, in addition to the studies involving free boundary surfaces in the unit ball, investigations of this kind have also been conducted in other domains. 
When the ambient space is a wedge (López \cite{Lopez:2014}), a slab (Ainouz and Souam \cite{Aiounz:2016}), a cone (Choe \cite{Choe:2011}) or a cylinder (Lopez and Pyo \cite{Lopez-Pyo:2014}). 
 
Regarding rigidity conclusions starting from a gap condition,   Andrade, Barbosa, and Pereira \cite{Andrade:2021} established some results for balls conforming to the Euclidean ball. More recently, when the ambient space is a strictly convex domain in a $3$-dimensional Riemannian manifold with
sectional curvature bounded above, and $\Sigma$ is a CMC free boundary surface in this region, Min and Seo  \cite{Min:2022} establish a pinching condition on the length of the umbilicity tensor on $\Sigma$. This criterion ensures that the surface is topologically equivalent to a disk or an annulus. In the particular case where the domain is a geodesic ball of a $3$-dimensional space form, they concluded that $\Sigma$ is a spherical cap or a Delaunay surface. In \cite{BarbosaFreitas:2023}, the first author, jointly with Barbosa, Melo, and Vitório, investigated the existence of compact free boundary
minimal hypersurfaces immersed in domains whose boundary is a regular level set, in particular giving some gap results for free boundary minimal hypersurfaces immersed in an Euclidean ball and
in a rotational ellipsoid.

% Specifically, they proved that if $\Sigma$ is a compact free boundary CMC surface in $(\mathbb{B}^3_r,\Bar{g}$, where $\Bar{g}= e^{2h}\langle,\rangle$, is conformal to the Euclidean metric, and for all points $x \in \Sigma$, 
% % \[
% \left\{
% \begin{aligned}
%    \frac{|\phi|^2}{\sigma^2}\Bar{g}(\vec{x},N)^2&\leq \frac{1}{2}\left(2+\frac{H}{\sigma}\Bar{g}(\vec{x},N)\right)^2\\
%     0&\leq 2+\frac{H}{\sigma}\Bar{g}(\vec{x},N),
% \end{aligned}
% \right.
% \]
% then $\Sigma$ is diffeomorphic to a disk or is rotationally symmetric with nontrivial topology.

This work explores some gap results for CMC free boundary surfaces in rotation domains described below. By considering a curve $\alpha(t)=(f(t),t)$, where $f$ is a positive real-valued smooth function, we generate a hypersurface $\partial\Omega$ starting from the revolution of this curve in an appropriate axis. In this sense, we can describe a domain $\Omega$ such that $\partial \Omega\subset F^{-1}(1)$ is a revolution hypersurface and $F:\mathbb{R}^n\times I\rightarrow \mathbb{R}$ is a smooth function given by
$$F(x, y) = \frac{1}{2}\left(|x|^2-f(y)^2\right) + 1. $$
Furthermore, we consider a hypersurface $\Sigma$, which is a free boundary CMC surface in $\Omega$. In our first results, we use an auxiliary function $g$ given by 
$$g(x,y)=\langle \Bar{\nabla}F,N\rangle$$
to get a topological characterization for CMC free boundary surfaces in $(n+1)$-dimensional rotation domains, with $n\geq 2$.
In particular, we obtain the following gap result for CMC surfaces in these domains
\begin{theorem}
\label{gap rotacional cmc n=3}
Let $\Sigma^2$ be a compact CMC surface with a free boundary in $F^{-1}(1)$. If  $(f')^2+ff''+1\leq 0$ and
$$|\phi|^2g(x,y)^2\leq\frac{1}{2}(2+2Hg(x,y))^2$$
on $\Sigma$, then $\Sigma$ is homeomorphic to a disk or an annulus.
\end{theorem}
Some considerations about the last result are necessary in this point. The inequality condition in $f$ has an interesting interpretation in terms of the principal curvatures of the profile curve, and the equality case is characterized by the euclidean ball (see Remarks \ref{explaining_fcond} and \ref{equality_ball}). In particular, this condition implies that the boundary of the domain is convex (see demonstration of Lemma \ref{lemma weingarten}). About the existence of free boundary minimal disks in convex regions of $\mathbb{R}^3$, we refer Struwe \cite{struwe1984} when he shows that there is at least one such disk. More recently,  Haslhofer and Ketover \cite{Ketover:2023} show these regions admit at least two free boundary minimal disks. On the other hand, by studying the existence of free boundary minimal annuli inside convex subsets of $3$-dimensional Riemannian manifolds with nonnegative Ricci curvature we refer \cite{Maximo:2017} by Máximo, Nunes, and Smith.  Also, we note that Theorem 3 is equivalent to \cite{Ambrozio:2021} and \cite[Theorem 1.3]{Barbosa:2023} when $\Omega$ is a ball. In fact, the rigidity statement in these last ones is expected: the free boundary CMC disks in the ball are totally umbilical (by Nitsche's result \cite{nitsche1985}). However, not all free boundary CMC annuli in the ball are catenoids or Delaunay surfaces: this is proved in recent papers by Fernandez- Hauwirth-Mira
(\cite{Fernandez:2023})
and Cerezo-Fernandez-Mira (\cite{Cerezo:2023}). In this sense, in the ball case, it shows that there exist free boundary minimal and CMC annuli that do not satisfy the gap inequality.

Talking about the higher dimensional case, we can prove the following result for minimal free-boundary hypersurfaces.

\begin{theorem}\label{gaphighdim}
Let $\Sigma^n$ be an $n$-dimensional free boundary minimal hypersurface in
a domain $\Omega$ with boundary $\partial\Omega\subset F^{-1}(1)$. Assume that $(f')^2+ff''+1\leq 0$. If
\begin{align*}
    |A|^2 g(x,y)^2\leq \frac{n}{n-1},
\end{align*}
for every $(x,y)\in\Sigma^n$, then one of the following is true:\\
1. $\Sigma^n$ is diffeomorphic to a disk $\mathbb{D}^n$.\\
2. $\Sigma^n$ is diffeomorphic to $\mathbb{S}^1\times\mathbb{D}^{n-1}$ and $C(\Sigma^n)$ is a closed geodesic.
\end{theorem}

The last part of the paper is dedicated to construct new examples of CMC surfaces that are free boundary in a particular rotational domain:  the rotational ellipsoid. By exploring the construction of Delaunay surfaces of Kenmotsu \cite{Kenmotsu:1980} (see also \cite[Section 4]{Barbosa:2023}), we show some examples of catenoids, nodoids, and onduloids that are free-boundary in that domain and satisfy the gap condition of the previous results.

This paper is organized as follows. In the second section, we approach some preliminaries for the theme and obtain some auxiliary lemmas that permit obtaining the main results of Section 3, about CMC free boundary surfaces in $3$-dimensional rotation domains and  minimal free boundary surfaces in $(n + 1)$-dimensional rotation domains. Finally, in the last section, we get some examples of Delaunay surfaces that are free boundary and satisfy our pinching condition in the particular case where the ambient space is a rotational ellipsoid.

\section{Preliminaries}\label{sec:Background}

Throughout this paper, we will consider $\Omega\subset\mathbb{R}^{n+1}$, with $n\geq 2$, be a rotation domain with smooth boundary $\partial\Omega\subset F^{-1}(1)$ where $F:\mathbb{R}^n\times I\rightarrow\mathbb{R}$ is a smooth function for some interval $I\subset\mathbb{R}$. We denote by $\Bar{N}:=\frac{\bar{\nabla} F}{|\bar{\nabla} F|}$ the outward
unit normal to $\partial\Omega$. Let $\Sigma^{n}\hookrightarrow\Omega$ an hypersurface with boundary such
that $\partial\Sigma\subset\partial\Omega$. We denote $N$ the outward unit normal to $\Sigma$ and $\nu$ the outward conormal along $\partial\Sigma$ in $\Sigma$. In this scope, a hypersurface $\Sigma$ is
called \emph{free boundary} if $\Sigma$ meets $\partial\Omega$ orthogonally, there is, $\nu = \Bar{N}$ along $\partial\Sigma$ or, equivalently, $\langle N, \Bar{N}\rangle = 0$ along $\partial\Sigma$.

More specifically, for $n=2$, let us consider a rotational hypersurface in the following sense. Let $\alpha(t) =(f(t), t)$ be a plane curve $\alpha$ that is the graph of a positive real valued smooth function $f : I \rightarrow \mathbb{R}$ in the $x_1x_3$-plane. Let $\Theta$ be a parametrization of the unit circle $\mathbb{S}^1$ in the plane $x_3 = 0$. The surface of revolution with generatriz $\alpha$ can be parametrized by
$$X(\Theta, t) = (\Theta f(t), t)=(\cos\theta f(t),\sin\theta f(t),t).$$
In this scope, we study free boundary surfaces $\Sigma$ in domains $\Omega$ which boundary is a surface of revolution given above.

Let us also consider $F : \mathbb{R}^2\times I \rightarrow \mathbb{R}$ be the smooth function defined by
\begin{align}
   \label{def F}
   F(x, y) = \frac{1}{2}\left(|x|^2-f(y)^2\right) + 1, 
\end{align}
where $x = (x_1,x_2)$ and $y = x_3$, we have that $\partial\Omega\subset F^{-1}(1).$ Notice that $1$ is a regular value of $F$.

Observe that
$$\bar{\nabla}F(x,y)=(x,-f(y)f'(y))=(x,y)+(0,-y-f(y)f'(y)),$$
where $y=\langle (x,y),E_3\rangle.$ Then,
\begin{align*}
    D^2F= \begin{pmatrix} 1 & 0 & 0 \\ 0 & 1 & 0\\ 0 & 0 & -(f'(y))^2-f(y)f''(y)  \end{pmatrix}=\text{Id}_{3\times 3}+ \begin{pmatrix} 0 & 0 & 0 \\ 0 & 0 & 0\\ 0 & 0 & -(f'(y))^2-f(y)f''(y)-1  \end{pmatrix}.
\end{align*}
Therefore, for all $X,Y \in T(\Sigma)$ we have
\begin{align}
\label{hess}
    \text{Hess}_\Sigma F(X,Y)&=\langle\bar{\nabla}_X(\bar{\nabla} F)^\top,Y\rangle\nonumber\\
    &=\langle\bar{\nabla}_X(\bar{\nabla} F -\langle\bar{\nabla} F,N\rangle N),Y\rangle\nonumber\\
    &=D^2F(X,Y)+\langle\bar{\nabla} F,N\rangle\langle A_NX,Y\rangle\nonumber\\
    &=\langle X,Y\rangle+g(x,y)\langle A_NX,Y\rangle-((f'(y))^2+f(y)f''(y)+1)\langle TX,Y\rangle,
\end{align}
where $T:T_{(x,y)})\Sigma\rightarrow T_{(x,y)}\Sigma$ is given by $TX=\langle X,E_3^\top\rangle E_3^\top$ and
$$g(x,y)=\langle\bar{\nabla} F,N\rangle=\langle(x,y),N\rangle+\langle N,E_3\rangle(-y-f(y)f'(y)).$$
It is straighfoward observe that, since the outward unit normal to $\partial\Omega$ is given by $\bar{N}=\frac{\bar{\nabla}F}{\bar{\nabla}F},$ the free boundary condition ensures that
$$g(x,y)=\langle\bar{\nabla} F,N\rangle=0\quad\mbox{on}\quad\partial\Sigma.$$
Furthermore, it is easy to check that $T$ is a self adjoint operator whose $E_3^\top$ is an eigenvector
associated with the eigenvalue $|E_3^\top|^2$. Besides, we can take any nonzero vector in
$T\Sigma$, orthogonal to $E_3^\top$, to verify that zero is also an eigenvalue of $T$. Therefore
\begin{align}
    \label{positivo definidoo}
    0 \leq \langle TX, X\rangle \leq |E_3^\top|^2|X|^2,\ \forall X \in T_{(x,y)}\Sigma.
\end{align}

From now one, we use the following condition for the function of the profile curve

\begin{equation}\label{ineq_f}
(f')^2+ff''+1\leq 0.   
\end{equation}

\begin{remark}\label{explaining_fcond}
We note that the condition \eqref{ineq_f} has an intriguing interpretation in terms of curvatures of the meridian and parallels of $\partial\Omega$. Indeed, in dimension $3$, the
principal curvatures of $\partial\Omega$ are
$$-\frac{f''}{(1+(f')^2)^{\frac{3}{2}}}\text{ and } \frac{1}{f\sqrt{1+(f')^2}},$$
which are attained by the meridians and the parallels, respectively. Then, the inequality $(f')^2+ff''+1\leq 0$ means that $\kappa_1\geq \kappa_2$ in $\Omega$, where $\kappa_1$ and $\kappa_2$ are the curvature of the meridians and parallels, respectively. Furthermore, this makes clear that $\kappa_{1}=\kappa_{2}$ for all $s$ and therefore, $\partial\Omega$ is a sphere, precisely in the equality case. 
\end{remark}

\begin{remark}
\label{equality_ball}
In an alternative way of the last remark, we observe that if $(f')^2+ff''+1=0$, we get
$$0=(f'(t))^2+f(t)f''(t)+1=(t+f(t)f'(t))'.$$
Then,
$$t+f(t)f'(t)=c_1,$$
where $c_1$ is a constant. Thus,
$$(f^2(t))'=2f(t)f'(t)=2(c_1-t)=(2c_1t-t^2)'.$$
Therefore,
$$f^2(t)=2c_1t-t^2+c_2,$$
where $c_2$ is a constant. It implies that
$$F(x,y)=\frac{1}{2}(|x|^2+y^2-2c_1y-c_2)+1.$$
Then, the set $F^{-1}(1)$ is the sphere
$$x_1^2+x_2^2+(y-c_1)^2=c_2+c_1^2.$$
\end{remark}

\begin{lemma}
\label{autovalores}
Suppose that $(f')^2+ff''+1\leq 0$. Then for each ${(x,y)}\in \Sigma$, the eigenvalues of Hess$_\Sigma F(x,y)$ are greater or equal to
$$1+k_1g(x,y) \text{ and } 1+k_2g(x,y),$$
where $k_1\leq k_2$ are the principal curvatures of $\Sigma$ with respect to the normal vector $N$.
\end{lemma}
\begin{proof}
Suppose that $(f')^2+ff''+1\leq 0$, then using (\ref{hess}) and (\ref{positivo definidoo}), we have that
\begin{align*}
\text{Hess}_\Sigma F(X,X)&=\langle X,X\rangle+g(x,y)\langle A_NX,X\rangle-((f'(y))^2+f(y)f''(y)+1)\langle TX,X\rangle\\
&\geq \langle X+g(x,y)AX,X\rangle.
\end{align*}
But, the eigenvalues of $X\rightarrow X+g(x,y)AX$ are 
$$1+k_1g(x,y) \text{ and } 1+k_2g(x,y),$$
where $k_1\leq k_2$ are the eigenvalues of A. Then, $k_1$ and $k_2$ are the principal curvature of $\Sigma$ and the eigenvalues $\lambda_1\leq\lambda_2$ of Hess$_\Sigma F(x,y)$  satisfy that
$$\lambda_1\geq1+k_1g(x,y) \text{ and } \lambda_2\geq1+k_2g(x,y).$$
\end{proof}

\section{Gap results}\label{sec:3 dimensional}

This section aims to give a topological classification of CMC free boundary hypersurfaces in the rotational domains, as defined earlier. We employ a gap condition in the umbilicity tensor and the graph function whose rotation generates the boundary domain. We subdivide our analysis in the three-dimensional and higher dimensional cases.

\subsection{CMC Free boundary surfaces in \texorpdfstring{$3$-dimensional}{3-dimensional} rotational domains}
In this subsection, we get a topological characterization for CMC free boundary surfaces in $3$-dimensional rotation domains. 

The next proposition shows that the gap condition given bellow implies the convexity of $F$ on $\Sigma$, and the proof of the result follows the same steps as in \cite[Lemma 2.1]{Barbosa:2023}.
\begin{proposition}
\label{Hess positiva}
Let $\Sigma$ be a compact free boundary CMC surface in $\Omega$. Assume that $(f')^2+ff''+1\leq 0$ and for all points (x,y) in $\Sigma$,
\begin{align}
    \label{hipotese de phi}
    |\phi|^2g(x,y)^2\leq\frac{1}{2}(2+2Hg(x,y))^2,
\end{align}
where $\phi=A-H\langle\cdot,\cdot\rangle$ is the umbilicity tensor. Then,
$$Hess_\Sigma F(X,X)\geq 0,$$
for all $(x,y)\in \Sigma$ and $X\in T_{(x,y)}\Sigma.$
\end{proposition}
\begin{proof}
By Lemma \ref{autovalores}, we have that the eigenvalues $\lambda_1\leq\lambda_2$ of Hess$_\Sigma F(x,y)$  satisfy that
$$\lambda_1\geq1+k_1g(x,y):=\tilde{\lambda_1} \text{ and } \lambda_2\geq1+k_2g(x,y):=\tilde{\lambda_2},$$
where $k_1$ and $k_2$ are the principal curvature of $\Sigma$. In order to prove Hess$_\Sigma F(X,X)\geq 0$, we need to show that $\lambda_1$ and $\lambda_2$ are nonnegative. Using condition (\ref{hipotese de phi}) we have
\begin{align}
    \label{eq lambda/gap}4\tilde{\lambda_1}\tilde{\lambda_2}&=4(1+k_1g(x,y))(1+k_2g(x,y))\nonumber\\
    &=4+4k_2g(x,y)+4k_1g(x,y)+4k_1k_2g(x,y)^2\nonumber\\
    &=4+8Hg(x,y)+4H^2g(x,y)^2+2(2H^2-|A|^2)g(x,y)^2\\
    &=(2+2Hg(x,y))^2-2|\phi|^2g(x,y)^2\geq0\nonumber.
\end{align}
Therefore, we need to show that at least one $\tilde{\lambda_i}$ is non-negative. For this, we will to show that function $v$ defined on $\Sigma$ and given by
$$v := \tilde{\lambda_1} + \tilde{\lambda_2} = 2 + 2Hg(x,y)$$
is nonnegative.
Note that we can assume that $\Sigma$ is not totally umbilical, otherwise it is obvious to check. Let us suppose that $v(p) < 0$ at some point $p\in\Sigma$. The free boundary condition ensures that
$$v = 2 + 2Hg(x,y)= 2$$
along $\partial\Sigma$. Choose $q\in\partial\Sigma$ and let $\alpha: [0, 1] \rightarrow \Sigma$ be a continuous curve such that $\alpha(0) = p$ and $\alpha(1) = q$. Since $v$ changes the signal along $\alpha$, there is a point $p_0 = \alpha(t_0),\ t_0\in(0, 1)$ such that $v(p_0) = 0$. In particular $g(x,y)(p_0)\neq 0$. Therefore, the condition (\ref{hipotese de phi}) implies that 
$$|\phi|^2(p_0)=0,$$
and hence $p_0$ is an umbilical point. Since $\Sigma$ is not a totally umbilical surface, we have that $p_0$ is an isolated point. So there is $\epsilon>0$ such that $v(\alpha(t)) < 0$, if $t\in[t_0-\epsilon, t_0)$ and $v(\alpha(t)) > 0$, if $t \in(t_0, t_0 +\epsilon]$, or vice-versa.

Let $D_{r_0}(p_0)$ be a geodesic disk with radius $r_0$ centered at $p_0$ such that $p_0$ is the
only umbilical point of $\Sigma$ on $D_{r_0}(p_0)$. We can choose $r_0$ and $\epsilon$ in such way that
$\alpha(t)\in D_{r_0}(p_0)$ for all $t\in [t_0-\epsilon, t_0 +\epsilon]$. Choose $\bar{r_0}<r_0$ such that $\alpha(t_0 -\epsilon),\ \alpha(t_0 +\epsilon)\notin D_{\bar{r_0}}(p_0)$. Let $\mathcal{A} = D_{r_0}(p_0) \setminus D_{\bar{r}_0}(p_0)$ be the annulus determined by these two discs and let $\beta$ denote a path in $\mathcal{A}$ joining the points $\alpha(t_0 -\epsilon)$ and $\alpha(t_0 +\epsilon)$. Again, $v$ changes the signal along of $\beta$, and therefore there is a point $\tilde{q}\in D_{r_0}(p_0)$
such that $v(\tilde{q}) = 0$. But, as above, it implies that $\tilde{q}$ is another umbilical point in $D_{r_0}(p_0)$ which is a contradiction and we conclude that $v \geq0$ as desired.

Then, $\lambda_i\geq\tilde{\lambda_i}\geq0$ for all $i$. Therefore $Hess_\Sigma F(X,X)\geq 0.$
\end{proof}

\begin{remark}
\label{obs vale a volta}
    Observe that, from (\ref{eq lambda/gap}) if we want to prove that the gap (\ref{hipotese de phi}) is valid, it is enough to show that $\tilde{\lambda_i}$ is non-negative for $i=1,2.$
\end{remark}

\begin{lemma}
\label{lemma weingarten}
Suppose that $(f')^2+ff''+1\leq 0$. Then the Weingarten operator $A^{\mathbb{R}^3}_{\partial\Omega}$ of $F^{-1}(1) = \partial\Omega$ in $\mathbb{R}^3$ with respect to inward unit normal satisfies
$$\langle A^{\mathbb{R}^3}_{\partial\Omega}X,X\rangle\geq k_1|X|^2>0,\ \forall X\in T\partial\Omega,\ x\neq0.$$

\end{lemma}
\begin{proof}
We claim that both eigenvalues $k_1 \leq k_2$ of $A^{\mathbb{R}^3}_{\partial\Omega}$ are positive. Let $U\subset\mathbb{R}^2$ be an open set and $x : U \subset\mathbb{R}^2\rightarrow V \subset\partial\Omega$ the immersion 
$$x(\theta, t) = (\cos \theta f(t),\sin \theta f(t),t),\ (\theta, t) \in U.$$
A straightforward calculation shows that the Gaussian curvature of $\partial\Omega$ at $x(\theta, t)$ is
$$K(\theta, t) =-\frac{ff''}{(1+(f')^2)^2f^2}>0.$$
Hence, $K$ is strictly positive on $\partial\Omega$. In particular, $k_1$ and $k_2$ have the same sign.
Furthermore, a simple calculation gives us 
$$H=\frac{1+(f')^2-ff''}{2f(1+(f')^2)^{\frac{3}{2}}}>0.$$
Therefore $k_2 > 0$ and $k_1 > 0$.
Thus, for all $X\in  T \partial\Omega$ with $X\neq0$,
$$\langle A^{\mathbb{R}^3}_{\partial\Omega}X,X\rangle\geq k_1|X|^2 > 0.$$
\end{proof}

% \begin{theorem}
% \label{gap rotacional cmc n=3}
% Let $\Sigma^2$ be a compact CMC surface with free boundary in $F^{-1}(1)$. If  $(f')^2+ff''+1\leq 0$ and
% $$|\phi|^2g(x,y)^2\leq\frac{1}{2}(2+Hg(x,y))^2$$
% on $\Sigma$, then $\Sigma$ is homeomorphic to a disk or an annulus.
% \end{theorem}
Now, we are in conditions to prove Theorem \ref{gap rotacional cmc n=3}. The argument follows the ideas of \cite{Ambrozio:2021}.
\begin{proof}[Proof of Theorem \ref{gap rotacional cmc n=3}]. 
First, we claim that the geodesic curvature $k_g$ of $\partial\Sigma$ in $\Sigma$ is positive. In fact, given $X,Y\in T\partial\Sigma$, we have on $\partial\Sigma$ that
$$\nabla_{X}^{\mathbb{R}^3}Y=\nabla_X^{\partial\Omega}Y+\langle A_{\partial\Omega}^{\mathbb{R}^3} X,Y\rangle\bar{N}=\nabla_X^{\partial\Sigma}Y+\langle A_{\partial\Sigma}^{\partial\Omega} X,Y\rangle N+\langle A_{\partial\Omega}^{\mathbb{R}^3} X,Y\rangle\bar{N}$$
and 
$$\nabla_{X}^{\mathbb{R}^3}Y=\nabla_X^{\Sigma}Y+\langle A_{\Sigma}^{\mathbb{R}^3} X,Y\rangle N=\nabla_X^{\Sigma}Y+\langle A_{\partial\Sigma}^{\Sigma} X,Y\rangle \bar{N}+\langle A_{\Sigma}^{\mathbb{R}^3} X,Y\rangle N.$$
Then, we will have $A^\Sigma_{\partial\Sigma}= A^{\mathbb{R}^3}_{\partial\Omega}$ on $\partial\Sigma$, where $\partial\Omega= F^{ -1}(1)$. Hence, if $X\in T \partial\Sigma$ is unitary, it follows from Lemma \ref{lemma weingarten} that 
\begin{align}
\label{Curvatura geodésica}
   k_g = \langle A^\Sigma_{\partial\Sigma}X,X\rangle=\langle A^{\mathbb{R}^3}_{\partial\Omega}X,X\rangle> 0.
\end{align}

Now, observe that if either $\Sigma$ is totally umbilical or $\Sigma$ has nonnegative Gaussian
curvature everywhere, then $\Sigma$ is homeomorphic to a disk. In fact, if $\Sigma$ is totally umbilical, we have that the Gaussian curvature $K_\Sigma$ of $\Sigma$ satisfies
$$K_\Sigma= H ^2 \geq 0.$$
Then, in any case, $\Sigma$ has nonnegative Gaussian curvature everywhere. From
the Gauss-Bonnet theorem and \eqref{Curvatura geodésica} , it follows
$$\int_\Sigma K_\Sigma+\int_{\partial\Sigma}k_g=2\pi\mathcal{X}(\Sigma)>0,$$
which shows that
$$\mathcal{X}(\Sigma)=2-2\hat{g}-r>0,$$
where $\hat{g}$ and $r$ are respectively the genus and quantity connected components of $\Sigma$. Then, $\hat{g}=0$ and $r=1$. Therefore, $\mathcal{X}(\Sigma) = 1$, $\Sigma$ is orientable and has exactly one boundary component. Thus, $\Sigma$ is homeomorphic to a disk.

Therefore, from now on, let us assume that $\Sigma$ is not a totally umbilical surface
and has negative Gaussian curvature at some point of $\Sigma$ and consider $$\mathcal{C} = \{p\in\Sigma; F(p) = \min_{ x\in\Sigma} F(x)\}.$$ 

Given $p, q\in \mathcal{C}$, let $\gamma: [0, 1] \rightarrow\Sigma$ be a geodesic such that $\gamma(0) = p$ and $\gamma(1) = q$. It
follows from Proposition \ref{Hess positiva} $\text{Hess}_\Sigma F \geq 0$ on $\Sigma$. Then, 
$$\frac{d^2}{dt^2}(F\circ\gamma)=\text{Hess}_\Sigma F\left(\frac{d\gamma}{dt},\frac{d\gamma}{dt}\right) \geq 0$$
for all $t\in [0, 1]$. Since $p,q\in \mathcal{C}$, we have
$$\frac{d}{dt}(F\circ\gamma)(0)=\frac{d}{dt}(F\circ\gamma)(1)=0, $$
which implies that $F$ is constant on $\gamma$ by the maximum principle. Then, we conclude that $(F \circ\gamma)(t) \equiv \text{min}_{\Sigma} F.$
Therefore, $\gamma([0, 1]) \subset\mathcal{C}$ and $\mathcal{C}$ must be a totally convex subset of $\Sigma$. In particular, totally convex property of $\mathcal{C}$ also assures that $\gamma([0, 1]) \subset\mathcal{C}$ for all geodesic loop $\gamma: [0, 1]\rightarrow\Sigma$, based at a point $p\in\mathcal{C}$.
Moreover, using \eqref{Curvatura geodésica} we assures that each geodesic $\gamma$ which connect two points in $\mathcal{C}$ is completely
inside of $\Sigma$, that is, the trace of $\gamma$ does not have points of $\partial\Sigma$. Hence, $\mathcal{C}$ is contained in the interior of $\Sigma$.

Finally, we claim that $\Sigma$ is homeomorphic to either a disk or an annulus. To see this,
we divide into two cases:\\

Case $1$: $\mathcal{C}$ consists of a single point.

Case $2$: $\mathcal{C}$ contains more than one point.\\

For Case $1$, let $p\in\Sigma\setminus\partial\Sigma$ be the only point of $\mathcal{C}$. Suppose that there is a non-trivial homotopy class $[\alpha] \in\pi_1(\Sigma, p)$, then we can find a geodesic loop $\gamma: [0, 1]\rightarrow\Sigma$, $\gamma(0) = \gamma(1) = p$ with $\gamma\in [\alpha]$. But, since $\mathcal{C}$ is totally convex, $\gamma([0, 1])\subset\mathcal{C}$ and, in particular, $\mathcal{C}$ has more than one point, which is a contradiction. This implies that $\pi_1(\Sigma, p)$ is trivial. Thus,  $\Sigma$ is simply connected and we conclude that $\Sigma$ is homeomorphic to a disk.

For Case $2$, we may assume that $\Sigma$ is not homeomorphic to a disk. Given $p\in \mathcal{C}$ we can find a
geodesic loop $\gamma : [0,1]\rightarrow\Sigma$, $\gamma(0) = \gamma(1)=p$ belonging to a non-trivial homotopy class $[\alpha]\in \pi_1(\Sigma, p)$. The totally convexity of $\mathcal{C}$ ensures that $\gamma([0,1])\subset\mathcal{C}.$ 
We claim that $\gamma$ is a regular curve. Indeed, if $\gamma'(0)\neq \gamma'(1)$, we can choose $\epsilon_0 > 0$ small and for each $\epsilon<\epsilon_0$ consider the minimizing geodesic $\tilde{\gamma}_\epsilon$ joining $\gamma(1 - \epsilon)$ and $\gamma(0 + \epsilon)$.
Since $\mathcal{C}$ is totally convex and $\gamma\subset\mathcal{C}$, we conclude that $\tilde{\gamma}_\epsilon\in \mathcal{C}.$ Now, we can choose an nonempty open set $U\subset\{\tilde{\gamma}_\epsilon\}_{\epsilon<\epsilon_0}$ of $\mathcal{C}$. Thus, for any geodesic $\beta(t)\in U$,
$$0=\frac{d^2}{dt^2}(F\circ\beta)=\text{Hess}_\Sigma F\left(\frac{d\beta}{dt},\frac{d\beta}{dt}\right)\geq 0.$$
Therefore, $\text{Hess}_\Sigma F\left(\frac{d\beta}{dt},\frac{d\beta}{dt}\right)=0$ in $U$.
By the proof of Lemma \ref{autovalores} and Proposition \ref{Hess positiva}
$$0=\text{Hess}_\Sigma F(e_i,e_i)\geq 1+\langle\Bar{\nabla}F,N\rangle k_i\geq 0.$$
Then,
$$1+\langle\Bar{\nabla}F,N\rangle k_1=1+\langle\Bar{\nabla}F,N\rangle k_2=0,$$
and we get that $k_1=k_2$ in $U$. Thus, the open subset $U$ is totally umbilical, which shows that $\Sigma$ must be totally umbilical which is a contradiction. Therefore $\mathcal{C}$ has to be equal to the unique closed geodesic $\gamma$. Since $[\alpha]$ was chosen to be arbitrary, this implies that $\pi_1(\Sigma, p) \approx \mathbb{Z}$ and $\Sigma$ is homeomorphic to an annulus.

\end{proof}

\subsection{Minimal free boundary surfaces in \texorpdfstring{$(n+1)$-dimensional}{(n+1)-dimensional} rotational domains }\label{sec:n+1 dimensional}

In this subsection, let us consider a rotational hypersurface in the following sense. Let $\alpha(t) =(f(t), t)$ be a plane curve $\alpha$ that is the graph of a positive real valued smooth function $f : I \rightarrow \mathbb{R}$ in the $x_1x_{n+1}$-plane. Let $\theta$ be a parametrization of the $n-$dimensional unit sphere in the hyperplane $x_{n+1} = 0$. The hypersurface of revolution with generatriz $\alpha$ can be parametrized by
$$X(\theta, t) = (\theta f(t), t).$$
In this scope, we study minimal free boundary surfaces in domains $\Omega$ which boundary is a hypersurface of revolution. Let us denote $x = (x_1,x_2,...,x_n)$ and $y = x_{n+1}$. Let $F : \mathbb{R}^{n+1} = \mathbb{R}^n\times\mathbb{R} \rightarrow \mathbb{R}$ be the smooth function defined by
$$F(x, y) = \frac{1}{2}\left(|x|^2-f(y)^2\right) + 1,$$
we have that $\partial\Omega\subset F^{-1}(1).$ Observe that, analogous to what was done in the previous section for dimension 3, denoting by $\Sigma$ a minimal free boundary surface in $\Omega$, we have
$$\nabla F(x,y)=(x,-f(y)f'(y))=(x,y)+(0,-y-f(y)f'(y)),$$
where $y=\langle (x,y),E_{n+1}\rangle.$ Then, for all $X,Y \in T(\Sigma)$ we have
\begin{align*}
    \text{Hess}_\Sigma F(X,Y)&=\langle X,Y\rangle+g(x,y)\langle A_NX,Y\rangle-((f'(y))^2+f(y)f''(y)+1)\langle TX,Y\rangle,
\end{align*}
where $T:T_{(x,y)})\Sigma\rightarrow T_{(x,y)}\Sigma$ is given by $TX=\langle X,E_{n+1}^\top\rangle E_{n+1}^\top$ and
$$g(x,y)=\langle\bar{\nabla} F,N\rangle=\langle(x,y),N\rangle+\langle N,E_{n+1}\rangle(-y-f(y)f'(y)).$$
We can write
\begin{align}
\label{hess n dimensional}
    \text{Hess}_\Sigma F(X,X)&=\langle X,X\rangle+\langle A(X,X),(\nabla F)^\bot\rangle-((f'(y))^2+f(y)f''(y)+1)\langle TX,Y\rangle,
\end{align}
It is easy to check that $T$ is a self adjoint operator whose $E_{n+1}^\top$ is an eigenvector
associated with the eigenvalue $|E_{n+1}^\top|^2$. Besides, we can take any nonzero vector in
$T\Sigma$, orthogonal to $E_{n+1}^\top$, to verify that zero is also an eigenvalue of $T$. Therefore
\begin{align*}
    0 \leq \langle TX, X\rangle \leq |E_{n+1}^\top|^2|X|^2,\ \forall X \in T_{(x,y)}\Sigma.
\end{align*}

\begin{lemma}\cite[Chen]{Chen:1973}
\label{lema algebrico}
Let $a_1,..., a_n$ and $b$ be real numbers. If
$$\sum_{i=1}^n a_i^2\leq\frac{(\sum_{i=1}^n a_i)^2}{n-1} -\frac{b}{n-1},$$
then $2a_ia_j\geq\frac{b}{n-1}$ for every $i, j\in\{1, . . . , n\}$.
\end{lemma}

The next proposition shows that the gap condition given bellow implies the convexity of $F$ on $\Sigma$.

\begin{proposition}
\label{Hess positiva dim alta}
Let $\Sigma^n$ be a minimal free boundary hypersurface $n$-dimensional in $\Omega$, with $n\geq3$. Assume that $(f')^2+ff''+1\leq 0$. If
\begin{align}
    \label{gap dimensão alta 1}
    |\nabla F^\bot|^2|A(x,y)|^2\leq \frac{n}{n-1},
\end{align}
for every $(x,y)\in\Sigma^n$. Then,
$$Hess_\Sigma F(X,X)\geq 0,$$
for all $(x,y)\in \Sigma$ and $X\in T_{(x,y)}\Sigma.$
\end{proposition}

\begin{proof}
Suppose that $(f')^2+ff''+1\leq 0$, then using (\ref{hess n dimensional}) we get
\begin{align}
    \label{hess caso particular}
    \text{Hess}_\Sigma F(X,X)&\geq\langle X,X\rangle+\langle A(X,X),(\nabla F)^\bot\rangle.
\end{align}
Let $\{e_1, . . . , e_n\}$ be an orthonormal basis of eigenvectors of $\text{Hess}_\Sigma F$ at $(x,y)\in\Sigma$ with respective eigenvalues $\lambda_1, . . . , \lambda_n$. We want to show
that $\lambda_i \geq 0$ for every $i$. By (\ref{hess caso particular}), $\lambda_i\geq \tilde{\lambda_i}:= 1+\langle A(e_i,e_i),(\nabla F)^\bot\rangle$ and joining with (\ref{gap dimensão alta 1}) we get
\begin{align*}
    \sum_{i=1}^n \tilde{\lambda_i}^2&=n+2\sum_{i=1}^n\langle A(e_i,e_i),(\nabla F)^\bot\rangle+\sum_{i=1}^n\langle A(e_i,e_i),(\nabla F)^\bot\rangle^2\\
    &\leq n+|\nabla F^\bot|^2\sum_{i=1}^n|A(e_i,e_i)|^2\\
    &\leq n++|\nabla F^\bot|^2|A|^2\leq n+\frac{n}{n-1}=\frac{n^2}{n-1}.
\end{align*}
On the other hand, we have that $( \sum_{i=1}^n \tilde{\lambda_i})^2=n^2$ since $\Sigma^n$ is minimal. Then
$$ \sum_{i=1}^n \tilde{\lambda_i}^2\leq \frac{( \sum_{i=1}^n \tilde{\lambda_i})^2}{n-1}.$$
By Lemma \ref{lema algebrico}, where $\tilde{\lambda_i} = a_i$ and $b = 0$, we get that $2\tilde{\lambda_i}\tilde{\lambda_j} \geq 0$. Consequently, the eigenvalues $\tilde{\lambda_i}$, $i = 1, . . . , n$, have all the same sign. Since $ \sum_{i=1}^n \tilde{\lambda_i}=n$, we conclude that $\tilde{\lambda_i} \geq 0$ for every $i$. Therefore, $\lambda_i\geq\tilde{\lambda_i}\geq0$ for every $i$. Then
$$Hess_\Sigma F(X,X)\geq 0,$$
for all $(x,y)\in \Sigma$ and $X\in T_{(x,y)}\Sigma.$
\end{proof}

%begin{definition}
%Given a $n$-dimensional free %boundary minimal hypersurface %$\Sigma^n$ in a domain $\Omega$ %with boundary %$\partial\Omega\subset F^{-1}(1)$. %We define
%$$C(\Sigma) = \{(x,y) \in\Sigma: %F(x,y) = m_0 := \text{min}_\Sigma %F\}.$$
%\end{definition}

\begin{proof}[Proof of Theorem \ref{gaphighdim}]
Firstly, let us define $\mathcal{C} = \{(x,y) \in\Sigma: F(x,y) = \text{min}_\Sigma F\}.$  From Proposition \ref{Hess positiva dim alta}, 
$$Hess_\Sigma F(X,X)\geq 0,$$
for all $(x,y)\in \Sigma$ and $X\in T_{(x,y)}\Sigma.$
The convexity of $Hess_\Sigma F$ strongly restricts the set $\mathcal{C}$ and the topology of $\Sigma$. We first prove that $\mathcal{C}$ is a totally convex set of $\Sigma$. As in the proof of Theorem \ref{gap rotacional cmc n=3}, the convexity of $\text{Hess}F$ restricted to $\Sigma$ implies that the set $\mathcal{C}$ is a totally convex set of $\Sigma$.

From now on, the proof follows the same line as in \cite[Theorem 3.7]{barbosaviana2} that uses standard Morse's theory. 
If $\mathcal{C}= \{(x_0,y_0)\}$, for some $(x_0,y_0)\in\Sigma$, $\Sigma$ is diffeomorphic to a disk $\mathbb{D}^n.$ If $\mathcal{C}$ contains more than one point we can show that $\text{dim}(\mathcal{C}) = 1$ and $\mathcal{C}$ is a geodesic. In this case, $\mathcal{C}$ is not a closed geodesic (what would imply that $\Sigma$ is diffeomorphic to a disk) or is a closed geodesic (what would force $\Sigma$ to be diffeomorphic to $\mathbb{S}^1\times\mathbb{D}^{n-1}$).

\end{proof}

\section{Examples of CMC free boundary surfaces in the rotational ellipsoid}

In this section, we show that there are a catenoid and some portions Delaunay surfaces that are free boundary on the rotational ellipsoid
\begin{align}
    \label{elipsoide rotacional}
    a^2x^2+a^2y^2+b^2z^2=R^2,
\end{align}
with $a^2\leq b^2$ and some constant $R^2$, and satisfy the pinching condition (\ref{hipotese de phi}).

\begin{remark}
\label{obs para f no elipsoide}
    Let us consider 
    $$f(y)=\frac{b}{a}\sqrt{\left(\frac{R}{b}\right)^2-y^2},$$
in (\ref{def F}). Then, we obtain the rotational ellipsoid given by (\ref{elipsoide rotacional}). In this case, the hypothesis $(f')^2+ff''+1\leq 0$ is automatically satisfied. In fact, we have
$$f'(y)=-\frac{yb}{a\sqrt{\left(\frac{R}{b}\right)^2-y^2}}$$ 
and
$$f''(y)=-\frac{R^2}{ab\left(\left(\frac{R}{b}\right)^2-b^2\right)^{\frac{3}{2}}}.$$
Therefore,
\begin{align*}
    (f')^2+ff''+1&=-\frac{y^2b^2}{a^2\left(\left(\frac{R}{b}\right)^2-y^2\right)}+\frac{b\sqrt{\left(\frac{R}{b}\right)^2-y^2}}{a}\left(-\frac{R^2}{ab\left(\left(\frac{R}{b}\right)^2-b^2\right)^{\frac{3}{2}}}\right)+1\\
    &=\frac{(a^2-b^2)\left(\left(\frac{R}{b}\right)^2-y^2\right)}{a^2\left(\left(\frac{R}{b}\right)^2-y^2\right)}=\frac{a^2-b^2}{a^2}\leq 0.
\end{align*}

%Furthermore, 
 %$$x'(s)(y+ff')=x'(s)z(s)\left(1-\frac{a^2}{b^2}\right).$$
 %Then, $x'(s)(y+ff')\leq 0$ if $x'(s)z(s)\geq 0.$
\end{remark}

First, let us consider a smooth curve parametrized by arc length in
the $xz$-plane $\beta(s) = (x(s), 0, z(s))$, with $x(s) > 0$ and denote by $\Sigma$ the surface obtained by rotation
of $\beta$ around the $z$-axis.

We start presenting a lemma with sufficient conditions for a general
rotational surface to satisfy the pinching condition (\ref{hipotese de phi}) in the rotational ellipsoid.

\begin{lemma}
\label{lema do gap}
Suppose that the curve $\beta$ satisfies the following conditions
\begin{align}
    \label{h1}
    -1\leq x''(s)\left(x(s)-\frac{x'(s)}{z'(s)}z(s)\frac{b^2}{a^2}\right),\ \text{if } z'(s)\neq0, 
\end{align}
\begin{align}
    \label{h2}
    -1\leq z(s)z''(s)\frac{b^2}{a^2},\ \text{if } z'(s)=0,
\end{align}
\begin{align}
    \label{h3}
    -x(s)x'(s)^2\leq z'(s)x'(s)z(s)\frac{b^2}{a^2},
\end{align}
with $a^2\leq b^2$. Then, $\Sigma$ satisfies the pinching condition
$$ |\phi|^2g(x,y)^2\leq\frac{1}{2}(2+2Hg(x,y))^2,$$
on the rotational ellipsoid given in (\ref{elipsoide rotacional}).
\end{lemma}

\begin{proof}
From Remark \ref{obs vale a volta} , it suffices to show that
$$\Tilde{\lambda}_1=1+k_1g(x,y)\geq 0 \text{ and }\title{\lambda}_2=1+k_2g(x,y)\geq0.$$
Let us consider $X:[s_1,s_2]\times \mathbb{S}^1\rightarrow \mathbb{R}^3$ given by
$$X(s,\theta)=(x(s)\cos(\theta),x(s)\sin(\theta),z(s)),$$
obtained by rotation of $\beta$ around the $z$-axis. Therefore,
$$X_s(s,\theta)=(x'(s)\cos(\theta),x'(s)\sin(\theta),z'(s))$$
and
$$X_\theta(s,\theta)=(-x(s)\sin(\theta),x(s)\cos(\theta),0).$$
Then, a straightforward computation shows that
$$N=(-z'(s)\cos(\theta),-z'(s)\sin(\theta),x'(s)).$$
Thus,
\begin{align}
\label{posição com N}
    \langle (x,y),N\rangle&=-x(s)z'(s)\cos^2(\theta)-x(s)z'(s)\sin^2(\theta)+x'(s)z(s)\nonumber\\
    &=-x(s)z'(s)+x'(s)z(s).
\end{align}
From (\ref{posição com N}) and Remark \ref{obs para f no elipsoide}, we get 
\begin{align}
\label{g(x,y)}
    g(x,y)&=\langle \nabla F,N\rangle\nonumber\\
    &=\langle (x,y),N\rangle -\langle N,E_3\rangle(y+f(y)f'(y))\nonumber\\
    &=-x(s)z'(s)+x'(s)z(s) -x'(s)\left(z(s)-\frac{b^2}{a^2}z(s)\right)\nonumber\\
    &=-x(s)z'(s)+x'(s)z(s)\frac{b^2}{a^2}. 
\end{align}
A straight forward computation shows that
\begin{align}
    \label{curvaturas k1 e k2}
    k_1=x'(s)z''(s)-x''(s)z'(s)\text{ and }k_2=\frac{z'(s)}{x(s)}.
\end{align}
If $z'(s)\neq 0,$ we can write 
\begin{align}
    \label{curvatura k1 z' não nulo}
    k_1(s)=-\frac{x''(s)}{z'(s)}.
\end{align}
Then, using (\ref{g(x,y)}) and (\ref{h1})
\begin{align*}
\Tilde{\lambda}_1&=1+k_1g(x,y)\\  
&= 1-\frac{x''(s)}{z'(s)}\left(-x(s)z'(s)+x'(s)z(s)\frac{b^2}{a^2}\right)\\
&=1+x''(s)\left(x(s)-\frac{x'(s)}{z'(s)}z(s)\frac{b^2}{a^2}\right)\geq0. 
\end{align*}
If $z'(s)=0,$ since the curve is parameterized by the arc length, then $x'(s)^2=1$. Using (\ref{g(x,y)}) and (\ref{h2}), we get
\begin{align*}
\Tilde{\lambda}_1&=1+k_1g(x,y)\\  
&= 1+\left(x'(s)z''(s)-x''(s)z'(s)\right)\left(-x(s)z'(s)+x'(s)z(s)\frac{b^2}{a^2}\right)\\
&=1+z''(s)z(s)\frac{b^2}{a^2}\geq0. 
\end{align*}
Finally, using again that the curve is parameterized by the arc length, together with (\ref{g(x,y)}) and (\ref{h3}), we obtain
\begin{align*}
\Tilde{\lambda}_2(s)&= 1+k_2g(x,y)\\
&= 1+\frac{z'(s)}{x(s)}\left(-x(s)z'(s)+x'(s)z(s)\frac{b^2}{a^2}\right)\\
&=\frac{x(s)-x(s)z'(s)^2+z'(s)x'(s)z(s)\frac{b^2}{a^2}}{x(s)}\\
&=\frac{x(s)x'(s)^2+z'(s)x'(s)z(s)\frac{b^2}{a^2}}{x(s)}\geq0.
\end{align*}
Therefore, $\tilde{\lambda}_1(s)\geq0$ and $\tilde{\lambda}_2(s)\geq0$ as desired.
\end{proof}

The function 
\begin{align}
        \label{funçao rho}
        \rho(s)=x(s)-\frac{x'(s)}{z'(s)}z(s)\frac{b^2}{a^2}.
    \end{align}
that appears in (\ref{h1}) has an important
geometric meaning. In fact, if $\rho(s_ 0) = 0$, then we can proof that $\Sigma$ is orthogonal to the rotational ellipsoid $E$ given by
    \begin{align*}
        a^2x^2+a^2y^2+b^2z^2=R^2,
    \end{align*}
where $R^2:=a^2x(s_0)^2+b^2z(s_0)^2.$ In particular we have the following lemma.

\begin{lemma}
\label{lemma Free boundary}
    Assume that $\beta(s)$ is defined for $s\in[c,d]$ and consider $\mathcal{Z} =\{s\in[c,d]; z'(s) = 0\}.$ Let us consider $a$ and $b$ positive real numbers, such that $a^2\leq b^2$ and define the function $\rho: [c,d] \setminus \mathcal{Z}\rightarrow\mathbb{R}$ by
    \begin{align*}
        \rho(s)=x(s)-\frac{x'(s)}{z'(s)}z(s)\frac{b^2}{a^2}.
    \end{align*}
    Let $s_1<s_2$ be two values in $[c,d]$ such that:\vspace{0.5cm}
    
    \text{(i)} $\rho(s_1)=\rho(s_2)=0$,\vspace{0.5cm}
    
    \text{(ii)} $a^2x(s_1)^2+b^2z(s_1)^2=a^2x(s_2)^2+b^2z(s_2)^2:=R^2$ and\vspace{0.5cm}

    \text{(iii)}  $a^2x(s)^2+b^2z(s)^2<R^2$ for all $s\in(s_1,s_2).$\vspace{0.5cm}
    
    Then, the rotation of $\beta_{|_{[s_1,s_2]}}$ produces a free boundary surface $\Sigma$ inside the rotational ellipsoid $E$ given by
    \begin{align}
        \label{elipsoide}
        a^2x^2+a^2y^2+b^2z^2=R^2.
    \end{align}
    
\end{lemma}

\begin{proof}
    The ellipsoid given in (\ref{elipsoide}) can be parametrized by    
    $$\Bar{X}(s,\theta)=\left(\frac{R}{a}\cos(s)\cos(\theta),\frac{R}{a}\cos(s)\sin(\theta),\frac{R}{b}\sin(s)\right).$$
A straight calculation show that
$$\Bar{N}=\frac{(\frac{1}{b}\cos(s)\cos(\theta),\frac{1}{b}\cos(s)\sin(\theta),\frac{1}{a}\sin(s))}{\sqrt{\frac{cos^2(s)}{b^2}+\frac{\sin^2(s)}{a^2}}}.$$
Now, observe that if $\rho(s_1)=\rho(s_2)=0$, then 
\begin{align*}
    0=\rho(s_i)=x(s_i)-\frac{x'(s_i)}{z'(s_i)}z(s_i)\frac{b^2}{a^2}.
\end{align*}
 We have, $z(s_i)\neq0$. In fact, if $z(s_i)=0$,  we conclude that $x(s_i)=0$, what does not happen. Thus, we can write
$$x'(s_i)=\frac{a^2}{b^2}\frac{x(s_i)}{z(s_i)}z'(s_i),$$
$i=1,2.$ Therefore,
\begin{align*}
    \beta'(s_i)&=\left(\frac{a^2}{b^2}\frac{x(s_i)}{z(s_i)}z'(s_i),0,z'(s_i)\right)\\
    &=z'(s_i)z(s_i)\frac{a}{b}\left(\frac{a}{b}x(s_i),0,\frac{b}{a}z(s_i)\right)
\end{align*}
On the other hand, using (ii) we have that the curve $\beta$ intersects the ellipsoid at the points $\beta(s_i)$. The normal at these points is given by
$$\Bar{N}(\beta(s_i))=\frac{\left(\frac{a}{b}x(s_i),0,\frac{b}{a}z(s_i)\right)}{R\sqrt{\frac{cos^2(s)}{b^2}+\frac{\sin^2(s)}{a^2}}}.$$
Then,
$$\beta'(s_i)=z'(s_i)z(s_i)\frac{a}{b}R\sqrt{\frac{cos^2(s)}{b^2}+\frac{\sin^2(s)}{a^2}}\Bar{N}(\beta(s_i)).$$
Thus, the rotation of $\beta_{|_{[s_1,s_2]}}$ is orthogonal to the ellipsoid in (\ref{elipsoide}). As by hypothesis we have $a^2x(s)^2+b^2z(s)^2<R^2$ for all $s\in(s_1,s_2)$ we get that $\Sigma\subset E.$
\end{proof}

Before presenting examples of CMC free boundary surfaces, let us introduce an example in the case where $H=0$, that is, a minimal free boundary surface in the rotational ellipsoid.

\begin{example}
    Consider $\Sigma$ the catenoid obtained by revolving the curve $\beta(s)=(\cosh(s),0,s)$ around the $z$-axis. Parameterizing by arc length we obtain the curve $\Bar{\beta}(s)=(\cosh(\sinh^{-1}(s)),0,\sinh^{-1}(s)).$ Taking $a^2=1$ and $b^2=2$ in (\ref{funçao rho}), we get that $\rho(s)=0$ if and only if
    $$\frac{1}{2\sinh^{-1}(s)}=\tanh(\sinh^{-1}(s)).$$
    Solving the equation we get that $s_1=-0,755...$ and $s_2=0,755...$ are such that $\rho(s_i)=0$ for $i=1,2.$ The parity of the functions $\cosh(s)$ and $\sinh^{-1}$ ensures that 
    $$(\cosh(\sinh^{-1}(s_1))^2+2(\sinh^{-1}(s_1))^2=(\cosh(\sinh^{-1}(s_2)))^2+2(\sinh^{-1}(s_2))^2,$$
    once $s_1=-s_2$. Then, let us define
$$R^2:=(\cosh(\sinh^{-1}(s_1))^2+2(\sinh^{-1}(s_1))^2=(\cosh(\sinh^{-1}(s_2)))^2+2(\sinh^{-1}(s_2))^2.$$
This way, the degrowth and growth of $\cosh(s)$ in $(s_1,0)$ and $(0,s_2)$, respectively, and the fact that $\sinh^{-1}(s)$ is increasing guarantee that  $(\cosh(\sinh^{-1}(s))^2+2(\sinh^{-1}(s))^2<R^2$ for all $s\in(s_1,s_2).$ Then, $\Sigma$ is a free boundary surface in the ellipsoid $E$ given by
$$x^2+y^2+2z^2=R^2.$$

Furthermore, with some calculations we get that
    \begin{align*}
       -1-x''(s)\left(x(s)-\frac{x'(s)}{z'(s)}z(s)\frac{b^2}{a^2}\right)=-1-\frac{1}{(1+s^2)^{\frac{3}{2}}}\left(\cosh(\sinh^{-1}(s))-2s\sinh^{-1}(s)\right)\leq0
    \end{align*}
and
    \begin{align*}
       -x(s)x'(s)^2-z'(s)x'(s)z(s)\frac{b^2}{a^2}=-\frac{s}{1+s^2}\left(s\cosh(\sinh^{-1}(s))+2\sinh^{-1}(s)\right)\leq 0,
   \end{align*}
for all $s\in[s_1,s_2].$ Then, from Lemma \ref{lema do gap}, $\Sigma$ satisfies the condition
$$ |\phi|^2g(x,y)^2\leq\frac{1}{2}(2+2Hg(x,y))^2.$$

%Now, to prove that $\Sigma$ satisfies the condition 
%$$ |\phi|^2g(x,y)^2\leq\frac{1}{2}(2+Hg(x,y))^2,$$ from Remark \ref{obs vale a volta}, it suffices to show that
%$$\Tilde{\lambda}_1=1+k_1g(x,y)\geq 0 \text{ and }\title{\lambda}_2=1+k_2g(x,y)\geq0.$$
% From (\ref{curvaturas k1 e k2}) and (\ref{curvatura k1 z' não nulo}), since $z'(s)=1\neq 0$ we get
 %$$k_1=-\cosh(s)<0 \text{ and } k_2=\frac{1}{\cosh(s)}>0.$$
 %Then, using the expression for $f$ of the ellipsoid given in  Remark \ref{obs para f no elipsoide} and and studying the sign of the functions $\cosh(s)$, $\sinh(s)$ and $s$, we get
 %$$\tilde{\lambda}_1(s)=1+\cosh^2(s)-2\cosh(s)\sinh(s)s>0\ \forall s\in(-0,97...,0,97...)$$
   %  In particular, $\tilde{\lambda}_1(s)>0$ for all $s\in (s_1, s_2).$ Futhermore,
%$$\tilde{\lambda}_2(s)=\frac{2s\sinh(s)}{\cosh(s)}\geq 0.$$
%Therefore, $\Sigma$ satisfies the condition
%$$ |\phi|^2g(x,y)^2\leq\frac{1}{2}(2+Hg(x,y))^2.$$
 \vspace{-0.2cm}
 \begin{figure}[H]
     \centering
         \includegraphics[width=10cm]{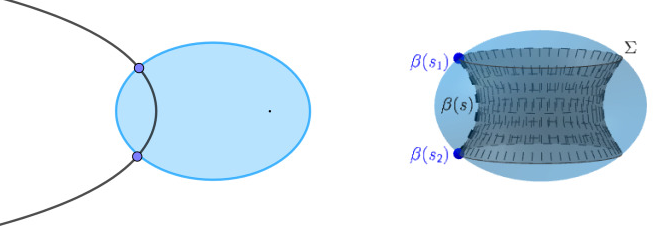}
         \vspace{-0.2cm}
     \caption{Catenoid free boundary in the elipsoid}
     
     \label{fig:Catenoid}
 \end{figure}

\end{example}

Now, let us consider a smooth curve parametrized by arc length in
the $xz$-plane $\beta(s) = (x(s), 0, z(s))$, with $x(s) > 0$, where
\begin{align}
    \label{x}
    x(s)=\frac{1}{H}\sqrt{1+B^2+2B\sin(Hs+\frac{3\pi}{2})}
\end{align}
and
\begin{align}
    \label{z}
    z(s)=\int^{s+\frac{3\pi}{2H}}_{\frac{3\pi}{2H}}\frac{1+B\sin(Ht)}{\sqrt{1+B^2+2B\sin(Ht)}}dt,
\end{align}
are given by the solution of
Kenmotsu \cite[Section 2, Equation (11)]{Kenmotsu:1980}, where $B, H \in \mathbb{R},$ with $ H > 0,\ B \geq 0$ and $B\neq 1$. Let denote by $\Sigma$ the surface obtained by rotation
of $\beta$ around the $z$-axis.

From Delaunay's Theorem, we know that any complete surface of revolution with constant mean
curvature is a sphere, a catenoid, or a surface whose generating curve is given by $\beta$. 

A surface whose generating curve is given by $\beta$ is called a Delaunay surface, with parameters $H$ and $B$, which can be of different types. If $B=0$ we get right cylinders. If $0 < B < 1$, Delaunay surfaces are embedded
and they are called unduloids. If $B > 1$ they are only immersed and called nodoids.

Observe that the components of the velocity vector of the curve $\beta(s)$ in the $xz$-plane are given by
\begin{align*}
    x'(s)=\frac{B\cos(Hs+\frac{3\pi}{2})}{\sqrt{1+B^2+2B\sin(Hs+\frac{3\pi}{2})}} \text{ and } z'(s)=\frac{1+B\sin(Hs+\frac{3\pi}{2})}{\sqrt{1+B^2+2B\sin(Hs+\frac{3\pi}{2})}}.
\end{align*}
And the acceleration components are given by
\begin{align*}
x''(s)=\frac{-BH(B+\sin(Hs+\frac{3\pi}{2}))(B\sin(Hs+\frac{3\pi}{2})+1)}{(1+B^2+2B\sin(Hs+\frac{3\pi}{2}))^\frac{3}{2}}
\end{align*}
and
\begin{align*}
z''(s)=\frac{HB^2\cos(Hs+\frac{3\pi}{2})(B+\sin(Hs+\frac{3\pi}{2}))}{(1+B^2+2B\sin(Hs+\frac{3\pi}{2}))^\frac{3}{2}}.
\end{align*}

Let us assume that $0 < B < 1$. The key observation in this case is that the function $z$ satisfies $z'(s) > 0$ for all $s$. Let $s_0$ be the smaller positive value such that $x''(s_0) = 0$. One can easily check that $s_0 = s_0(H, B) = \frac{1}{H}\sin^{-1}(-B) +\frac{\pi}{2H},$ where $\sin^{-1} : [-1, 1] \rightarrow [-\frac{\pi}{2},\frac{\pi}{2} ]$. Thus, given $s \in(-s_0, s_0)$ we have $z'(s) > 0$ and $x''(s) > 0.$

\begin{remark}
    In this case, we only have $x' = 0$ at point $0$, so the tangent is only vertical at this point. Therefore, we only have one wave of the unduloid inside the ellipsoid.
\end{remark}

Now, let us see some properties of the function $\rho$ that we will need later.
\begin{lemma}
\label{lema rho decrescente}
    Fix $0< B < 1$, $H > 0$, and consider the function $\rho : [-s_0, s_0] \rightarrow \mathbb{R}$ given by (\ref{funçao rho}). Then,\\
    i) $\rho(0)>0$.\\
    ii) $\rho'(0)=0$ and $\rho'(s_0)\leq0$.\\
    iii) $\rho$ is increasing in $(- s_0,0)$ and decreasing in $(0, s_0).$
\end{lemma}

\begin{proof}
   Observe that i) follows directly since
   \begin{align*}
       \rho(0)=x(0)-\frac{x'(0)}{z'(0)}z(0)\frac{b^2}{a^2}=\frac{1-B}{H}>0.
   \end{align*}
   To proof ii) we
observe that, since $\beta$ is parametrized by arc length, then
\begin{align*}
    \rho'(s)&=x'(s)-\frac{b^2}{a^2}\frac{((x'(s)z(s))'z'(s)-x'(s)z(s)z''(s))}{z'(s)^2}\\
    &=x'(s)+\frac{b^2}{a^2}\frac{x'(s)z(s)z''(s)-x''(s)z(s)z'(s)-x'(s)z'(s)^2}{z'(s)^2}\\
    &=\left(1-\frac{b^2}{a^2}\right)x'(s)+\frac{b^2}{a^2}z(s)\left(\frac{x'(s)z''(s)-x''(s)z'(s)}{z'(s)^2}\right)\\
    &=\frac{(a^2-b^2)}{a^2}x'(s)-\frac{b^2}{a^2}z(s)\left(\frac{x'(s)}{z'(s)}\right)'.
\end{align*}
As $x'(0)=0$ and $z(0)=0$ it follows that $\rho'(0)=0.$ On the other hand, using the expressions for $k_1$ given in (\ref{curvaturas k1 e k2}) and (\ref{curvatura k1 z' não nulo}) we get
\begin{align*}
     \rho'(s)&=\frac{(a^2-b^2)}{a^2}x'(s)-\left(\frac{-x'(s)z''(s)+x''(s)z'(s)}{z'(s)^2}\right)z(s)\frac{b^2}{a^2}\\
     &=\frac{(a^2-b^2)}{a^2}x'(s)+\frac{k_1(s)}{z'(s)^2}z(s)\frac{b^2}{a^2}\\
     &=\frac{(a^2-b^2)}{a^2}x'(s)-\frac{x''(s)}{z'(s)^3}z(s)\frac{b^2}{a^2}.
\end{align*}
Then, since $x''(s_0)=0$ we have that
\begin{align*}
    \rho'(s_0)&=(a^2-b^2)x'(s_0)\\
    &=\frac{(a^2-b^2)}{a^2}\frac{B\cos(Hs_0+\frac{3\pi}{2})}{\sqrt{1+B^2+2\sin(Hs_0+\frac{3\pi}{2})}}\\
    &=\frac{(a^2-b^2)}{a^2}\frac{B\sqrt{1-B^2}}{{\sqrt{1-B^2}}}\\
    &=\frac{(a^2-b^2)}{a^2}B\leq0.
\end{align*}
Finally, since $x''(s)>0$ and $x'(0)=0$ we get that $x'(s)>0$ for all $s\in (0,s_0)$ and $x'(s)<0$ for all $s\in(-s_0, 0).$ In the same way, we have $z(s)>0$ in $(0,s_0)$ and $z(s)<0$ in $(-s_0,0)$, then we obtain
\begin{align*}
    \rho'(s)=\frac{(a^2-b^2)}{a^2}x'(s)-\frac{x''(s)}{z'(s)^3}z(s)\frac{b^2}{a^2}<0
\end{align*}
in $(0,s_0),$ and
\begin{align*}
    \rho'(s)=\frac{(a^2-b^2)}{a^2}x'(s)-\frac{x''(s)}{z'(s)^3}z(s)\frac{b^2}{a^2}>0
\end{align*}
in $(-s_0,0)$. Therefore, $\rho$ is increasing in $(- s_0,0)$ and decreasing in $(0, s_0).$
\end{proof}

The next lemma gives us conditions to have an unduloid that is a free boundary surface on the rotational ellipsoid.
\begin{lemma}
\label{free boundary cmc no E}
    Fix $0 < B < 1$, $H > 0$, and set $z_0 = \frac{1-B^2}{HB}$. If $z(s_0) \geq z_0$, then $\rho(\Bar{s}) = 0$ for some $\bar{s} \in(0, s_0]$. In particular, the surface obtained by rotation of $\beta|_{[-\bar{s},\bar{s}]}$ is a free boundary CMC surface inside the rotational ellipsoid $E$ given by
    \begin{align*}
        a^2x^2+a^2y^2+b^2z^2=\Bar{R}^2,
    \end{align*}
where $\Bar{R}^2:=a^2x(\bar{s})^2+b^2z(\bar{s})^2$.
\end{lemma}

\begin{proof}
    If $z(s_0) \geq z_0$, then we get
    \begin{align*}
        \rho(s_0)&=x(s_0)-\frac{x'(s_0)}{z'(s_0)}z(s_0)\frac{b^2}{a^2}\\
        &\leq x(s_0)-\frac{x'(s_0)}{z'(s_0)}z_0\frac{b^2}{a^2}\\
        &=\frac{(a^2-b^2)}{a^2}\frac{\sqrt{1-B^2}}{H}\leq 0.
    \end{align*}
By assertion i) of Lemma \ref{lema rho decrescente}, $\rho(0) > 0$, and then  by continuity there is $\bar{s}\in (0\ s_0]$ such that $\rho(\bar{s}) = 0$. Using the parity of functions $\sin(Ht+\frac{3\pi}{2})$ and $\sin(Ht)$, we get
\begin{align*}
  x(-s)=x(s),\  x'(-s)=-x'(s),\ z(-s)=-z(s) \text{ and } z'(-s)=z'(s),
\end{align*} 
and thus,
$$\rho(-\bar{s}) = \rho(\bar{s}) = 0.$$
Moreover, $x'(0) = 0$ and $x''(s) > 0$ imply that $x'(s) > 0$ for all $s \in (0, \bar{s}]$. Therefore, $x'(s) > 0$ and $z'(s) > 0$ in $(0,\Bar{s})$, and it
ensures $a^2x^2(s) + b^2z^2(s) < \Bar{R}^2 := a^2x^2(\bar{s}) + b^2z^2(\bar{s})$ for all $s \in(0, \bar{s}]$. Because the curve
$\beta$ is symmetric with respect to $x$-axis  we  get $a^2x^2(s) + b^2z^2(s) \leq \Bar{R}^2$ for all $s\in[-\bar{s},\bar{s}]$ and we conclude that the surface is free boundary by Lemma \ref{lemma Free boundary}.
\end{proof}

\begin{example}
    Fix $B = 0,9$ and $H = 0,1$, so we have $z_0 =\frac{1-B^2}{HB}= 2,111...$ and $s_0 = 10\sin^{-1}(-0,9) + 5\pi \approx 4,51026.$ Then, we get
$$z_0(s_0)=\int_{15\pi}^{4,51026+15\pi}\left(\frac{1+(0,9)\sin(0,1t)}{\sqrt{1+(0,9)^2+(1,8)\sin(0,1t)}}\right)dt\approx2,71697.$$
    Therefore, $z(s_0)\geq z_0$. From Lemma \ref{free boundary cmc no E}, there is $\bar{s} \in(0, s_0]$ such that the surface obtained by rotation of $\beta|_{[-\bar{s},\bar{s}]}$ is a free boundary CMC surface inside the rotational ellipsoid $E$ given by
    \begin{align}
    \label{elipsoide2}
        a^2x^2+a^2y^2+b^2z^2=\Bar{R}^2,
    \end{align}
where $\Bar{R}^2:=a^2x(\bar{s})^2+b^2z(\bar{s})^2$.

\vspace{-0.5cm}
 \begin{figure}[H]
     \centering
         \includegraphics[width=12cm]{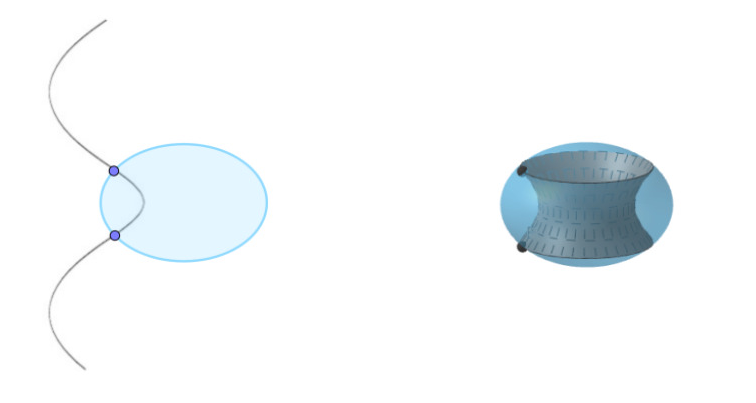}
         \vspace{-0.5cm}
     \caption{Unduloid free boundary in the elipsoid}
     
     \label{fig:Unduloid}
 \end{figure}

\end{example}

The next example says essentially that there are portions of unduloids that are free boundary in the ellipsoid given by (\ref{elipsoide2}) and satisfy the conditions of Lemma \ref{lema do gap}, there is, satisfy
$$ |\phi|^2g(x,y)^2\leq\frac{1}{2}(2+2Hg(x,y))^2,$$
on $\Sigma.$ 

\begin{example}
    Fix $0<B < 1$ and $H > 0$ and consider $\beta(s)=(x(s),0, z(s))$ as above and set $z_0=\frac{1-B^2}{HB}$. Let $s_0$ be the smaller positive value such that $x''(s_0) = 0$, in other words, $s_0= \frac{1}{H}\sin^{-1}(-B)+\frac{\pi}{2H}.$ Suppose $z(s_0)\geq z_0.$ From Lemma (\ref{free boundary cmc no E}), the surface $\Sigma$ obtained by rotation of $\beta|_{[-\Bar{s},\Bar{s}]}$, for some $\Bar{s}\in (0,s_0],$ is a free boundary CMC surface inside the rotational ellipsoid $E$ given by (\ref{elipsoide}). Moreover, in this case, for all $s\in[-\Bar{s},\Bar{s}]$ we have\vspace{0.5cm}
    
    (i) $x''(s)\geq0.$ In fact, we have $[-\Bar{s},\Bar{s}]\subset[-s_0,s_0],$ where $s_0$ was chosen to be the largest neighborhood of $0$ where $x''(s)\geq0.$\vspace{0.5cm}
    
    (ii) $\rho(s):=x(s)-\frac{x'(s)}{z'(s)}z(s)\frac{b^2}{a^2}\geq 0.$ Indeed, from Lemma \ref{free boundary cmc no E}, $\rho(\Bar{s})=0$. From Lemma \ref{lema rho decrescente}, $\rho$ is increasing in $(-s_0, 0)$ and decreasing in $(0, s_0)$. Therefore, $\rho(s)\geq0.$\vspace{0.5cm}
    
    (iii) $z(s)x'(s)\geq0.$ In fact, since $z'(s)>0$ and $x''(s)>0$ in $(-s_0,s_0)$, we get that $z$ and $x'$ are both growing in $(-s_0,s_0).$ Since $z(0)=x'(0)=0$, we conclude that $x'$ and $z$ has the same sign. \vspace{0.5cm}

The  items (i), (ii) and (iii) guarantee that the inequalities in Lemma \ref{lema do gap} are satisfied. In fact, from (i) and (ii) we get (\ref{h1}). Since $z'>0$, we do not need to show the validity of (\ref{h2}). Using that $x>0$ and (iii) we get (\ref{h3}). Therefore,
$$ |\phi|^2g(x,y)^2\leq\frac{1}{2}(2+2Hg(x,y))^2,$$
on $\Sigma.$

\end{example}

Now, let us assume that $B>1$. Let $r_0$ be the smaller positive value such that $z'(r_0)=0.$ We can check that $r_0=r_0(H,B)= \frac{1}{H}\sin^{-1}\left(-\frac{1}{B}\right)+\frac{\pi}{2H}$, where $\sin^{-1}:[-1,1]\rightarrow[-\frac{\pi}{2},\frac{\pi}{2}].$ In this case, we have $z'(r)<0$ and $x''(r)>0$ for all $r\in (-r_0,r_0)$.

\begin{remark}
    In this case, since $z' \neq 0$ for all $r\in(-r_0,r_0)$, we do not have horizontal tangents. Therefore, the node of the nodoids does not lie inside the ellipsoid.
\end{remark}

In the next Lemma we are going to show that there are portions of nodoids that are free boundary in the ellipsoid given by (\ref{elipsoide2}) and satisfy the conditions of Lemma \ref{lema do gap}, there is, satisfy
$$ |\phi|^2g(x,y)^2\leq\frac{1}{2}(2+2Hg(x,y))^2,$$
on $\Sigma.$ 

\begin{lemma}
\label{lema nodoide}
    Fix $B>1$ and $H>0$ and consider $\beta(r)=(x(r),0,z(r))$, with $x$ and $z$ given in (\ref{x}) and (\ref{z}), respectively. Let $r_0$ as above, then, there is $\bar{r}\in (-r_0,r_0)$ such that $\rho(\bar{r})=0$ and the surface obtained by rotation of $\beta|_{[-\bar{r},\bar{r}]}$ is a free boundary CMC surface inside the rotational ellipsoid $E$ given by
    \begin{align*}
        a^2x^2+a^2y^2+b^2z^2=\Bar{R}^2,
    \end{align*}
where $\bar{R}^2:=a^2x(\bar{s})^2+b^2z(\bar{s})^2$. 
Furthermore, we have 
$$ |\phi|^2g(x,y)^2\leq\frac{1}{2}(2+2Hg(x,y))^2,$$
on $\Sigma.$ 
\end{lemma}

\begin{proof}
In fact, we have that
$$\rho(0)=x(0)-\frac{x'(0)}{z'(0)}z(0)\frac{b^2}{a^2}=\frac{|1-B|}{H}>0,$$
and $\rho(r)\rightarrow -\infty$ when $r\rightarrow r_0$. Then, by continuity there is $\bar{r}\in(0,r_0)$ such that $\rho(\bar{r})=0.$ Using the parity of function $\rho$ we have $$\rho(\bar{r})=\rho(-\bar{r})=0.$$
Moreover, $x'(0) = 0$ and $x''(r) > 0$ imply that $x'(r) < 0$ for all $r \in (-\bar{r}, 0)$. Therefore, $x'(r) < 0$ and $z'(r) < 0$ in $(-\bar{r}, 0)$, and it
ensures $a^2x^2(r) + b^2z^2(r) < \Bar{R}^2 := a^2x^2(-\bar{r}) + b^2z^2(-\bar{r})$ for all $r \in(-\bar{r}, 0)$. Because the curve
$\beta$ is symmetric with respect to $x$-axis  we  get $a^2x^2(r) + b^2z^2(r) \leq \Bar{R}^2$ for all $r\in[-\bar{r},\bar{r}]$ and we conclude that the surface is free boundary by Lemma \ref{lemma Free boundary}. Furthermore, in this case, for all $r\in [-\bar{r}, \bar{r}]$ we have
\vspace{0.5cm}

(i) $\rho(r)\geq 0.$ Indeed, as already calculated in Lemma \ref{lema rho decrescente}, we have
$$\rho'(r)=\frac{(a^2-b^2)}{a^2}x'(r)-\frac{x''(r)}{z'(r)^3}z(r)\frac{b^2}{a^2}.$$
Since $x''(r)>0$ and $x'(0)=0$ we get that $x'(r)<0$ for all $r\in(-r_0, 0)$ and  $x'(r)>0$ for all $r\in (0,r_0)$. Similarly, we have $z(r)>0$ in $(-r_0,0)$ and $z(r)<0$ in $(0,r_0)$, then we obtain 
$\rho'(r)>0,\ \forall r\in (-\bar{r},0),$
and
 $\rho'(r)<0,\ \forall r\in (0,\bar{r}).$
Therefore, $\rho$ is increasing in $(- r_0,0)$ and decreasing in $(0, r_0).$ Since $\rho(0)>0$, we conclude that $\rho(r)\geq0, $ for all $r\in [-\bar{r}, \bar{r}]$.
\vspace{0.5cm}

(ii) $x'(r)z(r)\leq 0.$ In fact, since $z'(r)<0$ and $x''(r)>0$ in $(-r_0,r_0)$, we get that $x'$ is growing in $(-r_0,r_0)$ and $z$ is descending in $(-r_0,r_0)$. Since $z(0)=x'(0)=0$, we conclude that $x'$ and $z$ have opposite signs. \vspace{0.5cm}

The  items (i), (ii) and (iii) guarantee that the inequalities in Lemma \ref{lema do gap} are satisfied. In fact, from $x''(r)>0$ and (i) we get (\ref{h1}). Since $z'<0$, we do not need to show the validity of (\ref{h2}). Using that $x>0$ and (ii) we get (\ref{h3}). Therefore,
$$ |\phi|^2g(x,y)^2\leq\frac{1}{2}(2+2Hg(x,y))^2,$$
on $\Sigma.$

\end{proof}

\begin{example}
    Fix $B=1,1$ and $H=0,1$. Then, we have $r_0\approx10\sin^{-1}(-0.91)+5\pi\approx4,297...$. Therefore, $z'(r_0)=0$ and from Lemma \ref{lema nodoide}, there is $\bar{r} \in(0, r_0]$ such that the surface obtained by rotation of $\beta|_{[-\bar{r},\bar{r}]}$ is a free boundary CMC surface inside the rotational ellipsoid $E$ given by
    \begin{align*}
        a^2x^2+a^2y^2+b^2z^2=\Bar{R}^2,
    \end{align*}
where $\Bar{R}^2:=a^2x(\bar{r})^2+b^2z(\bar{r})^2$.

\vspace{0.3cm}
 \begin{figure}[H]
     \centering
         \includegraphics[width=12cm]{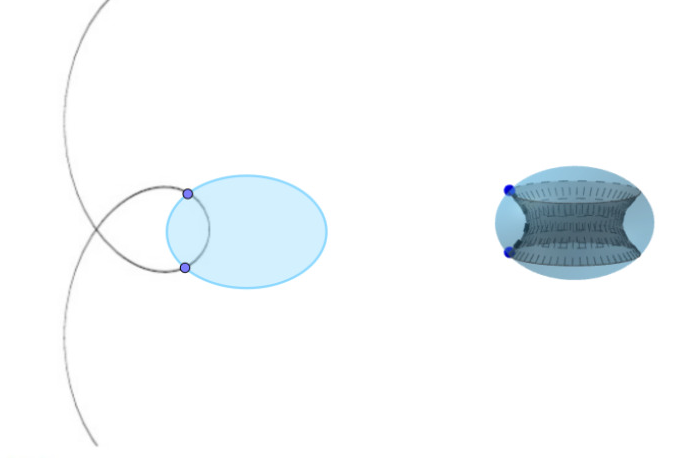}
         \vspace{-0.5cm}
     \caption{Nodoid free boundary in the elipsoid}
     
     \label{fig:nodoid}
 \end{figure}

\end{example}

\section*{Funding}
The first and second authors have been partially supported by Conselho Nacional de Desenvolvimento Científico e Tecnológico (CNPq) of the Ministry of Science, Technology and Innovation of Brazil,
Grants 316080/2021-7,  200261/2022-3, and 306524/2022-8. The authors  also were supported by Paraíba State Research
Foundation (FAPESQ), Grants 3025/2021 and  2021/3175 (A. Freitas).
The third author was partially supported by Grant 2022/1963, Paraíba State Research Foundation (FAPESQ).
% \subsection*{Authors' contributions}
% All authors wrote the main manuscript text and reviewed the manuscript.
\section*{Acknowledgements}
This work is a part of the Ph.D. thesis of the third author. The authors would like to thank Ezequiel Barbosa and Luciano Mari for their discussions about
the object of this paper and several valuable suggestions. 

The first author would like to thank
the hospitality of the Mathematics Department of Università degli Studi di Torino, where part of this work was carried out. The third author would like to express her gratitude for the hospitality and support during her visit to the Mathematics Department of Universidade Federal de Minas Gerais in April/May 2023. 

\section*{Data availability statement}
This manuscript has no associated data.

\end{document}